\documentclass[11pt,reqno,a4paper]{amsart}
\pdfoutput=1

\usepackage{latexsym,amsthm,amssymb,amsmath,amscd,mathtools,hyperref,xpatch}
\mathtoolsset{showonlyrefs}
\usepackage{tikz-cd}
\tikzcdset{
	cells={font=\everymath\expandafter{\the\everymath\displaystyle}},
}
\usetikzlibrary{arrows.meta,positioning,automata,calc}
\usetikzlibrary{matrix,arrows,decorations.pathmorphing}
\usepackage{bm}
\usepackage{fancyhdr}
\usepackage{mathtools}
\usepackage{mathrsfs} 
\usepackage{tcolorbox}
\usepackage{bbm}

\usepackage[normalem]{ulem}

\usepackage[all]{xy}

\usepackage{xpatch}

\makeatletter
\g@addto@macro\th@plain{\thm@postskip=0pt\relax}

\xpatchcmd{\proof}
{\topsep6\p@\@plus6\p@\relax}
{\topsep0pt\relax}
{}
{\PackageError{spacing}{Proof spacing patch failed}{}}
\makeatother

\usepackage[top=25truemm, bottom=20truemm,right=20truemm,left=20truemm]{geometry}

\usepackage{graphicx}

\usepackage[shortlabels]{enumitem}
\setlist[enumerate]{nosep}

\usepackage{xcolor}
\newcommand\redsout{\bgroup\markoverwith{\textcolor{red}{\rule[0.5ex]{2pt}{0.8pt}}}\ULon}

\newtheorem{theorem}{Theorem}[section]

\newtheorem*{conjecture*}{Conjecture}
\newtheorem*{theorem*}{Theorem}
\newtheorem{lemma}[theorem]{Lemma}
\newtheorem*{lemma*}{Lemma}

\newtheorem*{corollary*}{Corollary}
\newtheorem{proposition}[theorem]{Proposition}
\newtheorem{remark}[theorem]{Remark}

\newtheorem*{definition*}{Definition}
\newtheorem*{definitions*}{Definitions}

\newtheorem*{question*}{Question}

\newtheorem*{questions*}{Questions}

\newtheorem{thm}{Theorem}[section]

\newcommand{\cK}{\mathcal K}

\newcommand{\cN}{\mathcal N}

\newcommand{\cZ}{\mathcal Z}

\def\Cz{\mathbb{C}}
\def\Fz{\mathbb{F}}

\def\Nz{\mathbb{N}}

\def\Tz{\mathbb{T}}
\def\Zz{\mathbb{Z}}

\def\1z{\mathbb{1}}
\newcommand{\Tr}{{\rm Tr\,}}

\newcommand{\id}{{\rm id}}

\newcommand{\Ad}{\textup{Ad}}

\newcommand{\QT}{\textup{QT}}

\begin{document}
	
	\title{C*-selflessness and $\mathcal{Z}$-stability for Thompson's group $F$}
	
	\thispagestyle{fancy}
	
	\author{Chris Bruce}\thanks{This work was supported by the Engineering and Physical Sciences Research Council [grant number UKRI3684], by the Deutsche Forschungsgemeinschaft (DFG, German Research Foundation) under Germany’s Excellence Strategy EXC 2044-390685587, Mathematics M\"{u}nster: Dynamics–Geometry–Structure, and by the SFB 1442 of the DFG}
	\address[Chris Bruce]{School of Mathematics, Statistics and Physics, Herschel Building, Newcastle University, Newcastle upon Tyne, NE1 7RU, United Kingdom}
	\email{chris.bruce@newcastle.ac.uk}
	
	\author{Shirly Geffen}
	\address[Shirly Geffen]{Universität Münster, Mathematisches Institut, Einsteinstr. 62, 48149 Münster, Germany}
	\email{sgeffen@uni-muenster.de}

	\date{\today}
	
	\begin{abstract}
		We prove that the reduced C*-algebra of Thompson's group $F$ is $\cZ$-stable and selfless.
	\end{abstract}
	
	\maketitle
	
	\setlength{\parindent}{0cm} \setlength{\parskip}{0.3cm}
	
	\vspace{-1cm}
	\section{Introduction}
	The long-standing problem of whether or not Thompson's group $F$ is amenable was recently resolved by OpenAI, who proved that $F$ is nonamenable \cite{OpenAINonamenability} (see also the expository article by Kim and Lodha \cite{KimLodha}). OpenAI also proved the Kirchberg--Rørdam character criterion \cite{OpenAICharacter}. Here, we establish the following two
	applications of these remarkable results:
	\begin{thm}
		\label{thm:Z-stability}
		The C*-algebra $C_\lambda^*(F)$ is $\cZ$-stable.
	\end{thm}
	\begin{thm}
		\label{thm:selfless}
		The C*-algebra $C_\lambda^*(F)$ is selfless.   
	\end{thm}
    \vspace{-0.3cm}
	
	\subsection*{Nonamenability and operator algebras}
	The amenability problem for Thompson's group $F$ has long been
	closely connected with operator algebras. Lance's theorem identifies amenability of a discrete group with nuclearity of its reduced group C*-algebra \cite{LanceNuclear}. Thus, nonamenability of $F$ is equivalent to nonnuclearity of $C_\lambda^*(F)$. Kesten's spectral characterisation
	of amenability \cite{Kesten} also gives norm criteria in
	$C_\lambda^*(F)$, motivating the computational work of S.~Haagerup, U.~Haagerup, and Ramirez-Solano \cite{HaagerupHaagerupRamirez}.
	Their moment calculations yielded rigorous lower bounds and numerical evidence suggesting nonamenability, without establishing it.
    Concerning approximation by finite permutations, Chaudkhari, Juschenko, and Muliarchyk recently proved that Thompson's groups $F$, $T$, and $V$ are sofic, and hence hyperlinear \cite[Corollary~1.2]{CJM26}.
	
	Tracial structure provides another connection. The character
	classification of Dudko and Medynets \cite{DudkoMedynets} implies that $F$ is nonamenable if and only if $C_\lambda^*(F)$ has a unique tracial state; see \cite[Remark~5.6]{HaagerupOlesen}.
	
	The connection with C*-simplicity was established by Haagerup and Olesen, who proved that simplicity of the reduced C*-algebra $C_\lambda^*(T)$ of Thompson's group $T$ implies nonamenability of $F$ \cite[Theorem~5.5]{HaagerupOlesen}. Bleak and Juschenko obtained a partial converse in terms of ideal generation \cite{BleakJuschenko}. Le Boudec and Matte Bon subsequently proved the full converse and established the equivalences
	\[
	F \text{ is nonamenable}
	\quad\Longleftrightarrow\quad
	C_\lambda^*(F) \text{ is simple}
	\quad\Longleftrightarrow\quad
	C_\lambda^*(T) \text{ is simple}
	\]
	\cite[Corollary~4.2]{LBMB}.
	
	Further progress concerns quasi-regular representations and ideal structure. Kalantar and Scarparo proved that the C*-algebra $C^*_{\lambda_{T/F}}(T)$ generated by the quasi-regular representation on $\ell^2(T/F)$ is simple and has no tracial states \cite[Theorem~6.1]{KalantarScarparo}.
	Brix, the first-named author, Li, and Scarparo proved that $C_\lambda^*(T)$ always has a unique maximal ideal, and that it has infinitely many ideals exactly when $F$ is amenable \cite[Theorem~4.6 and Corollary~4.9]{BBLS}.
	In that case, its unique nonzero simple quotient is
	$C^*_{\lambda_{T/F}}(T)$ \cite[Remark~4.7]{BBLS}.
	
	Another strand of the operator-algebraic study of $F$ concerns central
	sequences. Jolissaint proved that the group von Neumann algebra
	$L(F)$ is McDuff: it absorbs the hyperfinite $\mathrm{II}_1$ factor
	$R$ tensorially \cite[Theorem~A]{Jolissaint}. This result holds
	independently of amenability and provides an earlier instance of
	tensorial absorption for an operator algebra associated with $F$.
	Theorem~\ref{thm:Z-stability} establishes absorption of the Jiang--Su algebra at the level of the reduced group C*-algebra. 
	
	\subsection*{Selfless C*-algebras}
	Robert introduced selflessness for C*-probability spaces
	\cite{RobertSelfless}; we recall the definition in
	Subsection~\ref{subsec:comparison-selflessness}. A discrete group
	$\Gamma$ is called \emph{C*-selfless} if
	$(C_\lambda^*(\Gamma),\tau_\Gamma)$ is selfless, where $\tau_\Gamma$
	is the canonical tracial state. In the tracial setting, selflessness
	implies simplicity, uniqueness of the tracial state, stable rank one,
	and strict comparison \cite[Theorem~3.1]{RobertSelfless}. Thus, C*-selflessness implies
	C*-simplicity (i.e., simplicity of $C_\lambda^*(\Gamma)$); whether the converse holds for nontrivial countable
	discrete groups remains open \cite[Problem~XCI]{STW26}.
	
	Amrutam, Gao, Kunnawalkam Elayavalli, and Patchell proved that finitely
	generated nonabelian free groups are C*-selfless, resolving the
	longstanding strict-comparison problem for $C_r^*(\Fz_n)$, $2\leq n<\infty$
	\cite{AGKP25}. Following work on higher-rank lattices \cite{Vig25},
	Ozawa established C*-selflessness for countable acylindrically
	hyperbolic groups with trivial finite radical \cite{Ozawa26}, Arzhantseva and Finn-Sell give a criterion for C*-selflessness
of lacunary hyperbolic groups with rapid decay
\cite[Theorem~4.4]{AFS26}, and
	Vigdorovich established it for nontrivial linear groups with trivial
	amenable radical \cite{Vig26}. Further developments concern
	commensurator groups of hyperbolic groups \cite{BF26}, boundary
	dynamics on hyperbolic spaces \cite{Rybak26}, and lattices in affine
	buildings \cite{LLV26}.
	
	Arimoto recently proved that Thompson's groups $T$ and $V$ are
	C*-selfless, using nonamenability of $F$ for the conclusion about
	$T$, and noted that C*-selflessness of $F$ remained open
	\cite{Arimoto}. In contrast to Theorem~\ref{thm:Z-stability},
	neither $C_\lambda^*(T)$ nor $C_\lambda^*(V)$ is $\cZ$-stable.
	Indeed, for $G\in\{T,V\}$, such stability would give an isomorphism
	$C_\lambda^*(G)\cong C_\lambda^*(G)\otimes\cZ$ approximately
	unitarily equivalent to the first-factor embedding
	\cite[Theorem~2.2]{TomsWinter}, hence preserving the canonical tensor
	trace. Since $\cZ$ is nuclear with a unique tracial state, its tracial GNS closure is an injective $\mathrm{II}_1$
factor and hence the hyperfinite $\mathrm{II}_1$ factor $R$ by Connes' theorem \cite{Connes}, so this would give
	$L(G)\cong L(G)\,\overline\otimes\,R$. Thus, $L(G)$ would be McDuff
	and have property $\Gamma$, contrary to Haagerup--Olesen
	\cite[Theorem~4.4]{HaagerupOlesen} (using Effros' theorem\cite{Effros}).
	
	The route to selflessness used here comes from Robert's recent
	converse to the implication from selflessness to strict comparison:
	if $D$ is an infinite-dimensional simple unital C*-algebra with a
	unique tracial state $\tau$ and strict comparison with respect to $\tau$,
	then $(D,\tau)$ is selfless \cite[Theorem~1.1]{RobertComparison}. We apply this theorem after
	establishing strict comparison for $C_\lambda^*(F)$.

	\subsection*{Outline of the proof of Theorem~\ref{thm:Z-stability}}
	Besides the nonamenability of $F$
	\cite{OpenAINonamenability}, the proof of Theorem~\ref{thm:Z-stability}
	uses a second result of OpenAI, proved alongside the
	Kirchberg--R{\o}rdam character criterion: the infinite minimal tensor power of every nonzero, unital, separable $C^*$-algebra without characters contains a unital copy of $\cZ$ \cite[Corollary~1.3]{OpenAICharacter}.
	
	Kirchberg and R{\o}rdam posed the central-sequence character problem in \cite[Question~3.1]{KirchbergRordamCharacters} and recorded the related
	infinite-tensor embedding question of Dadarlat and Toms in
	\cite[Question~2.6]{KirchbergRordamCharacters}. The earlier theorem of Dadarlat and Toms \cite[Theorem~1.1]{DadarlatToms} establishes $\cZ$-stability of the infinite tensor power under the additional hypothesis that it contains a unital subhomogeneous $C^*$-subalgebra
	without characters. The character problem also appears as
	Problem~LXXX in the survey of Schafhauser, Tikuisis, and White \cite{STW26}; footnote~137 there discusses the related tensor-power question.
	
	We first construct a unital copy of $\cZ$ in $C_\lambda^*(F)$
	(Proposition~\ref{prop:localZ}). Nonamenability of $F$ and
	Le Boudec--Matte Bon's theorem \cite[Corollary~4.2]{LBMB} give simplicity
	of $C_\lambda^*(F)$, which implies it has no characters; as $C_\lambda^*(F)$ is also unital and separable, $C_\lambda^*(F)^{\otimes\infty}$ contains a unital copy of $\cZ$. Copies of $F$ supported on infinitely many dyadic intervals with pairwise disjoint interiors commute with one another, and their reduced $C^*$-algebras realise $C_\lambda^*(F)^{\otimes\infty}$
	unitally inside $C_\lambda^*(F)$.
	
	We then turn this copy of $\cZ$ into a central sequence embedding (Proposition~\ref{prop:insertion}). After rescaling, we may assume that it lies in the $C^*$-subalgebra generated by the unitaries $\lambda_g$
	with $g$ supported in $I_0=[1/4,1/2]$. Conjugation by $\lambda_s^j$, where $s\in F$ is defined by
    \[s(t)=\begin{cases}
			t/2,&0\leq t\leq1/2,\\
			t-1/4,&1/2\leq t\leq3/4,\\
			2t-1,&3/4\leq t\leq1,
		\end{cases}\] 
    moves it to the analogous subalgebra for $I_j=2^{-j}I_0$. These intervals have pairwise disjoint interiors and
	accumulate at $0$. Every $g\in F$ agrees with $t\mapsto2^{r}t$ near $0$ for some $r=r(g)\in\Zz$. Hence, for all large $j$, conjugation by $\lambda_g$ carries the copy of $\cZ$ on $I_j$ to the copy on
	$I_{j-r}$. Equivariant $\cZ$-stability of the Bernoulli shift $\sigma$
	on $\bigotimes_{\Zz}\cZ$ \cite[Theorem~B]{GL} provides unital
	completely positive maps $\psi_n$ from $\cZ$ into finite tensor blocks
	that are asymptotically multiplicative and asymptotically
	$\sigma$-invariant (Lemma~\ref{lem:shiftblocks}). Placing the $n$th
	block on copies with indices at least $n$ gives unital completely
	positive maps $\rho_n\colon\cZ\to C_\lambda^*(F)$ such that, for each $g\in F$ and
	all large $n$,
	\[
	\|[\lambda_g,\rho_n(z)]\|=\|\sigma^{-r(g)}(\psi_n(z))-\psi_n(z)\|
	\qquad(z\in\cZ).
	\]
	By density, the images of the $\rho_n$ asymptotically commute with
	every element of $C_\lambda^*(F)$. The resulting unital embedding of $\cZ$ into the
	central sequence algebra $C_\lambda^*(F)_\infty\cap C_\lambda^*(F)'$ gives $C_\lambda^*(F)\cong C_\lambda^*(F)\otimes\cZ$
	by Toms--Winter's absorption criterion \cite[Theorem~2.2]{TomsWinter}.
	
	Theorem~\ref{thm:Z-stability} is a $C^*$-algebraic analogue of
	Jolissaint's theorem that $L(F)$ is a McDuff factor, that is,
	$L(F)\cong L(F)\,\overline\otimes\,R$ for the hyperfinite
	$\mathrm{II}_1$ factor $R$ \cite[Theorem~2.1]{Jolissaint}, and the two
	proofs run in parallel. Jolissaint works at the endpoint $1$, where $s$ acts as $t\mapsto2t-1$. The elements of $F$ that are
	trivial near $1$ form a subgroup $D$ with $F=D\rtimes\langle s\rangle$.
	Group elements supported close to $1$ give noncommuting central
	sequences for $L(D)$ \cite[Corollary~2.5]{Jolissaint}. Conjugation by
	powers of $s$ shifts suitable such sequences, so the induced
	$\Zz$-action on $L(D)$ is centrally free
	\cite[Corollary~2.7]{Jolissaint}. Ocneanu's theory of centrally free
	actions of amenable groups \cite{Ocneanu}, in the form of
	\cite[Proposition~2.6]{Jolissaint}, then shows that $L(F)$ is McDuff.
	In our proof, elements of $F$ that are trivial near $0$ commute with all
	sufficiently far-out copies of $\cZ$, and $s$ shifts these copies.
	Equivariant $\cZ$-stability of the Bernoulli shift takes the place of
	Ocneanu's theory: in both arguments, amenability of $\Zz$ yields copies
	of the absorbed algebra that are approximately invariant under the
	endpoint shift. There are two essential differences. First, Theorem~\ref{thm:Z-stability} requires central sequences in operator norm rather than in the $2$-norm induced by the trace, so the shift step needs a norm-level absorption theorem. Second, every $\mathrm{II}_1$ factor contains a unital copy of $R$, whereas a unital copy of $\cZ$ in $C_\lambda^*(F)$ forces $F$ to be nonamenable: for amenable $F$, the
	trivial representation would give a character on $C_\lambda^*(F)$, and hence on
	$\cZ$. Thus, Jolissaint's theorem does not depend on the amenability problem for $F$. 
	
	\subsection*{Outline of the proof of Theorem~\ref{thm:selfless}}
	For Theorem~\ref{thm:selfless}, the remaining step is to show that
	$\tau_F$ is the unique normalised quasitrace on $C_\lambda^*(F)$
	(Proposition~\ref{prop:uniqueQT}). We do not know whether $F$ is exact,
	so we cannot appeal to Haagerup's theorem that quasitraces on exact
	$C^*$-algebras are traces \cite[Theorem~5.11]{Haagerup}. We first work
	in $C_\lambda^*(F')$ (Proposition~\ref{prop:restriction}), where $F'=[F,F]$. Since elements of $F'$ are
	trivial near $0$ and $1$, every finite set $S\subseteq F'$ has a
	conjugate $gSg^{-1}$, with $g\in F'$, whose elements are supported away
	from those of $S$ and hence commute with them. Iterating gives
	arbitrarily many pairwise commuting conjugate copies of $S$. Averaging
	over $N$ such copies preserves quasitrace values and shrinks commutators
	by a factor of at least $N$. Approximate additivity of quasitraces
	\cite[Corollary~II.2.6]{BH} therefore forces every normalised
	quasitrace on $C_\lambda^*(F')$ to be a tracial state
	(Lemma~\ref{lem:commutingconjugates}). As $\tau_{F'}$ is the unique
	tracial state on $C_\lambda^*(F')$, every normalised quasitrace on $C_\lambda^*(F)$ restricts to
	$\tau_{F'}$ on $C_\lambda^*(F')$.
	
	The homomorphism $F\to\Zz^2$ given by $g\mapsto(\log_2g'(0),\log_2g'(1))$ has kernel $F'$, 
    and we use the associated splitting $F=F'\rtimes\Zz^2$ to pass from $C_\lambda^*(F')$ 
    to $C_\lambda^*(F)$. 
	Each element in the group algebra $\Cz[F]$ thus has a finite Fourier expansion over $\Zz^2$
	with coefficients in $\Cz[F']$. Conjugating one bump element of $F'$ by
	the powers of one generator of $\Zz^2$, and another by the powers of the
	other, gives elements with pairwise disjoint supports. They generate two
	commuting copies of $C(\Tz^{\Zz})$ in $C_\lambda^*(F')$, on which the two generators
	act by independent coordinate shifts and $\tau_{F'}$ restricts to
	product Haar measure (Lemmas~\ref{lem:endpoint-bumps}
	and~\ref{lem:endpoint-shift-algebras}). A similar two-endpoint structure
	appears in the remark following \cite[Corollary~2.7]{Jolissaint}, where shifted central sequences near $0$
	and $1$ show that the $\Zz^2$-action on $L(F')$ is centrally free. Here
	we use this structure quantitatively. Finite coordinate windows can be
	shifted towards $0$ and $1$, where they commute exactly with any
	prescribed finite subset of $\Cz[F']$
	(Lemma~\ref{lem:endpoint-localisation}). A continuous form of Rokhlin's
	lemma for the coordinate shift (Lemma~\ref{lem:tower}) supplies tower
	bases depending on finitely many coordinates, with exactly orthogonal
	translates whose squares have total trace close to one. Let $a$ be a
	positive element of the complex group-algebra and $0<\delta<1$. Placing such towers near the
	endpoints, where they commute with the relevant Fourier coefficients of
	$a$, and combining them with the positive matrix over $C_\lambda^*(F')$ formed from
	these coefficients gives an exact matrix compression
	(Lemma~\ref{lem:compression}): there are a positive contraction $c\in C_\lambda^*(F')$,
	an integer $d\geq1$ and a row $y\in M_{1,d}(C_\lambda^*(F))$ such that
	\[
	cac=yy^*,\qquad y^*y\in M_d(C_\lambda^*(F'))_+,\qquad
	\tau_F(1-c^2)<\delta.
	\]
	
	Let $q$ be a normalised quasitrace on $C_\lambda^*(F)$, and let $q_d$ be its matrix
	amplification to $M_d(C_\lambda^*(F))$. Since $q|_{C_\lambda^*(F')}=\tau_{F'}$, uniqueness of
	amplifications gives $q_d=\tau_{F'}\otimes\Tr$ on $M_d(C_\lambda^*(F'))$. Applying
	$q_d(YY^*)=q_d(Y^*Y)$ to the matrix $Y$ with first row $y$ yields
	$q(cac)=\tau_F(cac)$. As $1-c^2\in C_\lambda^*(F')$, also
	$q(1-c^2)=\tau_F(1-c^2)<\delta$. A remainder estimate based on
	square-root subadditivity of quasitraces (Lemma~\ref{lem:remainder})
	then gives $|q(a)-\tau_F(a)|\leq2\|a\|\sqrt\delta$. Letting
	$\delta\to0$ gives $q(a)=\tau_F(a)$. Since positive finitely supported elements
	are dense in $C_\lambda^*(F)_+$ and quasitraces are norm continuous, $q$ and $\tau_F$
	agree on $C_\lambda^*(F)_+$, and hence on $C_\lambda^*(F)$ (Proposition~\ref{prop:rigidity}).
	
	Finally, Theorem~\ref{thm:Z-stability} and
	\cite[Proposition~5.2]{APTV} give strict comparison of positive elements
	of $C_\lambda^*(F)$ by quasitraces. As $C_\lambda^*(F)$ is simple and unital, this comparison can
	be tested on normalised quasitraces \cite[Paragraph~3.4]{APTV}, so
	Proposition~\ref{prop:uniqueQT} gives strict comparison by $\tau_F$.
	Since $C_\lambda^*(F)$ is infinite-dimensional, simple, unital and monotracial,
	Robert's theorem \cite[Theorem~1.1]{RobertComparison} shows that
	$(C_\lambda^*(F),\tau_F)$ is selfless.
	
	\subsection*{Declaration on the use of AI}
	The authors used the large language models GPT-6 Astra and GPT-6.1 Sol from OpenAI and Claude Opus 5.5 and Claude Fable 5.1 from Anthropic in preparing this paper. These were used to produce first drafts of proofs of results, to suggest improvements to the exposition, and to proofread the manuscript. The authors have verified every proof in the paper and have rewritten many of the AI-drafted ones.
	
	CB is supported by a ChatGPT for Academic Researchers plan and a Claude Team plan for Scientists.

	\section{Preliminaries}
	
	\subsection{Thompson's group $F$}
	
	Let $F$ be Thompson's group of increasing piecewise linear
	homeomorphisms of $[0,1]$ with finitely many breakpoints, all in
	$\Zz[1/2]\cap(0,1)$, and slopes in $2^{\Zz}$. We refer the reader to
	\cite{CFP} and \cite{Belk} for background on $F$. We record the basic
	facts we need, in addition to the nonamenability of $F$ from
	\cite{OpenAINonamenability} (see also \cite{KimLodha}).
	
	We write $F'=[F,F]$ for the commutator subgroup of $F$, and put
	\[
	A=C_\lambda^*(F),\qquad B=C_\lambda^*(F').
	\]
	All tensor products are minimal unless explicitly indicated otherwise.
	For $g\in F$, write $\lambda_g$ for its left regular unitary in $C_\lambda^*(F)$.
	For every subgroup $H\leq F$, the restriction of the left regular
	representation of $F$ to $H$ decomposes over the right cosets $Hg$
	as copies of the left regular representation of $H$. Thus the
	canonical inclusion of group rings extends to an isometric unital
	embedding $C_\lambda^*(H)\subseteq C_\lambda^*(F)$.
	The canonical traces satisfy
	\[
	\tau_F(\lambda_g)=
	\begin{cases}1,&g=1,\\0,&g\ne1,\end{cases}
	\qquad \tau_{F'}=\tau_F|_B.
	\]
	For a nondegenerate closed interval $I\subseteq[0,1]$ with dyadic
	endpoints, set
	\[
	F_I=\{g\in F:g(t)=t\text{ for every }t\in[0,1]\setminus I\},
	\qquad A_I=C_\lambda^*(F_I)\subseteq C_\lambda^*(F).
	\]
	
	\begin{lemma}[{\cite[Theorem 4.1]{CFP}}]
		\label{lem:F'}
		The map
		\begin{equation}
			\label{eq:endpoint}
			\varphi:F\longrightarrow\Zz^2,\qquad
			\varphi(g)=(\log_2g'(0),\log_2g'(1))
		\end{equation}
		where the derivatives at $0$ and $1$ are one-sided, has kernel $F'$.
		Moreover, $F'$ consists of the elements of $F$ which are trivial in
		neighbourhoods of $0$ and $1$.
	\end{lemma} 
	
	We shall make use of the following basic property of $F\curvearrowright [0,1]$, which follows from \cite{CFP}.
\begin{lemma}[See {\cite[Lemma~4.2]{CFP}}]
\label{lem:interval-moving}
Let $K,J\subseteq(0,1)$ be nondegenerate closed intervals with
dyadic endpoints. Then, there exists $g\in F'$ with $g(K)=J$.
\end{lemma}
\begin{proof}
Write $K=[a,b]$ and $J=[c,d]$, where $a,b,c,d\in \Zz[1/2]\cap (0,1)$. Choose
dyadic points $\alpha,\beta\in \Zz[1/2]\cap (0,1)$ satisfying
\[
0<\alpha<\min\{a,c\}, \qquad \max\{b,d\}<\beta<1.
\]
Apply \cite[Lemma~4.2]{CFP} to the dyadic partitions
\[
(0,\alpha,a,b,\beta,1)
\quad\text{and}\quad
(0,\alpha,c,d,\beta,1)
\]
to obtain $h\in F$ sending each point of the first partition
to the corresponding point of the second. Define
\[
g(t)=
\begin{cases}
h(t), & t\in[\alpha,\beta],\\
t, & t\in[0,\alpha]\cup[\beta,1].
\end{cases}
\]
Since $h$ fixes $\alpha$ and $\beta$, the pieces agree
at the joining points and define an increasing homeomorphism.
All breakpoints are dyadic and all slopes are powers of two,
so $g\in F$. Moreover, $g$ is the identity near both
endpoints of $[0,1]$, hence belongs to $F'$ by
Lemma~\ref{lem:F'}. Finally, $g(a)=c$ and $g(b)=d$,
so $g(K)=J$.
\end{proof}

	We use the canonical element $s\in F$ defined by
	\begingroup
	\begin{equation}\label{eq:s-generator}
		s(t)=\begin{cases}
			t/2,&0\leq t\leq1/2,\\
			t-1/4,&1/2\leq t\leq3/4,\\
			2t-1,&3/4\leq t\leq1.
		\end{cases}
		\qquad
		\vcenter{\hbox{%
				\resizebox{!}{2.5cm}{%
					\begin{tikzpicture}[x=1.5cm,y=1.5cm,font=\scriptsize,
						line cap=round,line join=round]
						\draw[gray!55,densely dotted]
						(0,1/4) -- (1/2,1/4) -- (1/2,0)
						(0,1/2) -- (3/4,1/2) -- (3/4,0);
						\draw[gray!65,dashed] (0,0) -- (1,1);
						\draw[->] (0,0) -- (1.08,0) node[right] {$t$};
						\draw[->] (0,0) -- (0,1.08) node[above] {$s(t)$};
						\draw[thick] (0,0) -- (1/2,1/4) -- (3/4,1/2) -- (1,1);
						\fill (0,0) circle (0.7pt);
						\fill (1/2,1/4) circle (0.7pt);
						\fill (3/4,1/2) circle (0.7pt);
						\fill (1,1) circle (0.7pt);
						\node[below left] at (0,0) {$0$};
						\draw (1/2,0.012) -- (1/2,-0.012) node[below] {$\frac12$};
						\draw (3/4,0.012) -- (3/4,-0.012) node[below] {$\frac34$};
						\draw (1,0.012) -- (1,-0.012) node[below] {$1$};
						\draw (0.012,1/4) -- (-0.012,1/4) node[left] {$\frac14$};
						\draw (0.012,1/2) -- (-0.012,1/2) node[left] {$\frac12$};
						\draw (0.012,1) -- (-0.012,1) node[left] {$1$};
					\end{tikzpicture}
		}}}
	\end{equation}
	\endgroup
	We use the following explicit splitting (cf.\ \cite[Chapter C]{BS16}).
	Let $u$ and $v$ be the affine copies of $s$ and $s^{-1}$ on
	$[0,1/2]$ and $[1/2,1]$, respectively, extended by the identity:
	\[
	u(t)=
	\begin{cases}
		\frac12s(2t),&0\leq t\leq\frac12,\\
		t,&\frac12\leq t\leq1,
	\end{cases}
	\qquad
	v(t)=
	\begin{cases}
		t,&0\leq t\leq\frac12,\\
		\frac12\bigl(1+s^{-1}(2t-1)\bigr),&\frac12\leq t\leq1.
	\end{cases}
	\]
	These elements commute and satisfy
	\[
	\varphi(u)=(-1,0),\qquad \varphi(v)=(0,-1).
	\]
	Thus $\varphi$ is surjective and, by Lemma~\ref{lem:F'}, induces an
	isomorphism $F/F'\cong\Zz^2$. The map $(m,n)\mapsto u^{-m}v^{-n}$
	is a section of $\varphi$. We use the identification induced by
	$-\varphi$, so that $F=F'\rtimes\langle u,v\rangle$ and
	$u^{k_1}v^{k_2}$ represents $k=(k_1,k_2)\in\Zz^2$.
	With this convention, every element $a\in \Cz[F]$ has a unique expression of the form
	\begin{equation}
		\label{eq:fourier}
		a=\sum_{k\in\Zz^2} b_kW_k,
	\end{equation}
	where $W_k\coloneqq\lambda_{u^{k_1}v^{k_2}}$ and $b_k\in\Cz[F']$ are zero for all but finitely many $k$. 
	
	We record the following important consequence of the nonamenability of $F$, which follows from the theorem by Le Boudec and Matte Bon \cite[Corollary~4.2]{LBMB} combined with general results by Breuillard, Kalantar, Kennedy, and Ozawa \cite{BKKO}:
	\begin{theorem}
		\label{thm:LBMB}
		If $F$ is nonamenable, then both $C_\lambda^*(F)$ and $C_\lambda^*(F')$ are simple C*-algebras.
	\end{theorem}
	\begin{proof}
		By \cite[Corollary~4.2]{LBMB}, nonamenability of $F$ implies that
		$F$ is $C^*$-simple, that is, $C_\lambda^*(F)$ is simple.
		Since $F'$ is a normal subgroup of $F$,
		\cite[Theorem~1.4]{BKKO} implies that $F'$ is also $C^*$-simple.
	\end{proof}
	Every $C^*$-simple group has trivial amenable radical. Hence,
	\cite[Theorem~1.3]{BKKO} implies that $C_\lambda^*(F)$ and $C_\lambda^*(F')$ are monotracial,
	with unique tracial states $\tau_F$ and $\tau_{F'}$, respectively.

	\subsection{Quasitraces}\label{subsec:quasitraces}
	
	Throughout this subsection, $D$ denotes a unital $C^*$-algebra with
	unit $1$. By a \emph{quasitrace} on $D$ we mean a bounded
	$2$-quasitrace in the sense of Blackadar and Handelman~\cite{BH}; in the
	formulation of \cite[Definition~3.1]{Haagerup}, this is a function
	$q\colon D\to\mathbb{C}$ such that
	\begin{enumerate}
		\item[(Q1)] $q(x^*x)=q(xx^*)\geq 0$ for all $x\in D$;
		\item[(Q2)] $q(a+ib)=q(a)+i\,q(b)$ for all selfadjoint $a,b\in D$;
		\item[(Q3)] $q$ is linear on every abelian $C^*$-subalgebra of $D$;
		\item[(Q4)] there is a function $\tilde q\colon M_2(D)\to\mathbb{C}$
		satisfying (Q1)--(Q3) with $M_2(D)$ in place of $D$, such that
		$q(x)=\tilde q(x\otimes e_{11})$ for all $x\in D$.
	\end{enumerate}
    We write $\QT_1(D)$ for the
	set of \emph{normalised} quasitraces on $D$, that is, those with
	$q(1)=1$.
	
	We shall use the following facts, valid for every quasitrace $q$ on $D$.
	\begin{enumerate}
		\item[(a)] \emph{Order preservation} \cite[Corollary~3.4(1)]{Haagerup}:
		if $x,y\in D$ are selfadjoint and $x\leq y$, then $q(x)\leq q(y)$.
		\item[(b)] \emph{Matrix amplification} \cite[Corollary~3.4(2)]{Haagerup}:
		for every $n\geq 1$ there is a unique quasitrace $q_n$ on $M_n(D)$
		such that $q_n(x\otimes e_{11})=q(x)$ for all $x\in D$.
		\item[(c)] \emph{Square-root subadditivity}
		\cite[Lemma~3.5(1)]{Haagerup}: for all positive $a,b\in D$,
		\begin{equation}\label{eq:sqrt}
			\sqrt{q(a+b)}\leq\sqrt{q(a)}+\sqrt{q(b)}.
		\end{equation}
		\item[(d)] \emph{Approximate additivity} \cite[Corollary~II.2.6, p.~324]{BH}:
		for every $\varepsilon>0$ there is $\delta>0$ such that, for all
		selfadjoint contractions $x,y\in D$,
		\begin{equation}\label{eq:almostcommuting}
			\|[x,y]\|<\delta\quad\Longrightarrow\quad
			|q(x+y)-q(x)-q(y)|<\varepsilon,
		\end{equation}
		where $[x,y]=xy-yx$. Here, $\delta$ may depend on $q$ as well as
		on $\varepsilon$.
	\end{enumerate}
	
	Two further properties follow from the definition and~(a). First,
	$q$ is unitarily invariant: $q(uxu^*)=q(x)$ for all $x\in D$ and all
	unitaries $u\in D$. 
    Second, $q$ is bounded and norm continuous (cf.\
	\cite[Corollary~II.2.5, p.~324]{BH}). 
	
	\subsection{Jiang--Su absorption}\label{subsec:Z-stability}
	
	A $C^*$-algebra $D$ is $\cZ$-stable if $D\cong D\otimes\cZ$, where $\cZ$ is the Jiang--Su algebra \cite{JiangSu}.
	For a unital $D$, write
	\[
	D_\infty=\ell^\infty(\Nz,D)/c_0(\Nz,D),
	\]
	and identify $D$ with its diagonal image. The relative commutant
	$D_\infty\cap D'$ is its central sequence algebra.
	Since $\cZ$ is strongly self-absorbing
	\cite[Example~1.14(iv)]{TomsWinter}, for separable unital $D$ a unital
	$*$-homomorphism $\cZ\to D_\infty\cap D'$ implies $\cZ$-stability
	\cite[Theorem~2.2]{TomsWinter}.
	
	\subsection{Strict comparison and selflessness}\label{subsec:comparison-selflessness}
	
    Let $D$ be a simple separable unital $C^*$-algebra
	with unique tracial state $\tau$, and let $\cK$ denote the compact
	operators on a separable infinite-dimensional Hilbert space. For
	$c\in(D\otimes\cK)_+$, put
	\[
	d_\tau(c)=\lim_{n\to\infty}(\tau\otimes\Tr)(c^{1/n}),
	\]
	where $\tau\otimes\Tr$ is the lower semicontinuous extension of the
	unnormalised matrix traces. Strict comparison by $\tau$ means that
	\[
	d_\tau(a)<d_\tau(b)\quad\Longrightarrow\quad a\precsim_{\mathrm{Cu}}b
	\qquad(a,b\in(D\otimes\cK)_+).
	\]
	Here, $a\precsim_{\mathrm{Cu}}b$ means that there are $x_n\in D\otimes\cK$ with $\|a-x_n^*bx_n\|\to0$; this is Cuntz subequivalence.
	
	We recall the ultrapower formulation of selflessness
	from \cite[Section~1]{RobertComparison}. For a unital $D$ with faithful
	tracial state $\tau$, the pair $(D,\tau)$ is selfless if there exist a free
	ultrafilter $\mathcal U$ on $\mathbb N$, a unital $C^*$-algebra
	$C\ne\Cz$ with a state $\kappa$ whose GNS representation is faithful,
	and a unital state-preserving $*$-homomorphism
	\[
	(D,\tau)*(C,\kappa)\longrightarrow(D^{\mathcal U},\tau^{\mathcal U})
	\]
	whose restriction to $D$ is the diagonal embedding. The free product
	is reduced, and $D^{\mathcal U}$ is the norm ultrapower
	\[
	D^{\mathcal U}
	=\ell^\infty(\mathbb N,D)\big/
	\{(x_n):\lim_{n\to\mathcal U}\|x_n\|=0\},
	\qquad
	\tau^{\mathcal U}([(x_n)])=\lim_{n\to\mathcal U}\tau(x_n).
	\]
	
	\section{From the Kirchberg--R\o rdam character criterion to $\cZ$-stability}
	\subsection{A local copy of the Jiang--Su algebra}
	
	A basic dyadic interval has the form
	$[k2^{-n},(k+1)2^{-n}]$, where $n\geq0$ and $0\leq k<2^n$.
	If $I=[a,b]$ has dyadic endpoints and $b-a\in2^\Zz$, affine
	rescaling from $[0,1]$ onto $I$, followed by extension by the identity,
	gives a group isomorphism $F\to F_I$ and hence an isomorphism
	$A\to A_I$, where $A=C_\lambda^*(F)$ and we stick to the notation introduced in the preliminary section. Both rescaling and its inverse preserve dyadic
	breakpoints and slopes in $2^\Zz$; see the argument of
	\cite[Lemma~4.4]{CFP}.
	
	Finitely many subgroups supported on intervals with disjoint interiors form a direct product. They commute, and restricting a relation to each interval proves injectivity. The regular representation of this product is the spatial tensor product of its regular representations. Thus, their reduced $C^*$-algebras realise the minimal tensor product faithfully inside $C_\lambda^*(F)$. Common endpoints cause no difficulty since each factor fixes its interval's endpoints.
	
	\begin{proposition}\label{prop:localZ}
		$C_\lambda^*(F)$ contains a unital copy of $\cZ$.
	\end{proposition}
	
	\begin{proof}
		By the nonamenability input \cite{OpenAINonamenability} and
		Theorem~\ref{thm:LBMB} (by Le Boudec and Matte Bon), $C_\lambda^*(F)$ is simple, and thus has no characters.
		
		By \cite[Corollary~1.3]{OpenAICharacter}, the infinite minimal tensor power
		\[
		C_\lambda^*(F)^{\otimes\infty}
		=\varinjlim_N\bigl(C_\lambda^*(F)^{\otimes N},x\mapsto x\otimes1\bigr)
		\]
		contains a unital copy of $\cZ$. We embed this tensor power unitally into $C_\lambda^*(F)$.
		Choose countably many basic dyadic intervals with disjoint interiors
		inside $[1/4,1/2]$; for example, take
		\[
		J_n=[1/4+2^{-n-2},\,1/4+2^{-n-1}],\qquad n\geq1.
		\]
        Let $A_{J_n}$ be as introduced in the preliminary section.
		Let $\alpha_n:A\to A_{J_n}$ be the isomorphism induced by affine
		rescaling. The disjoint-interval construction gives isometric unital
		homomorphisms
		\[
		\iota_N:A^{\otimes N}\longrightarrow A_{[1/4,1/2]},\qquad
		\iota_N(a_1\otimes\cdots\otimes a_N)
		=\alpha_1(a_1)\cdots\alpha_N(a_N).
		\]
		Since $\alpha_{N+1}(1_A)=1_A$, these maps satisfy
		$\iota_{N+1}(x\otimes1_A)=\iota_N(x)$. They therefore induce
		an isometric unital homomorphism
		$\iota:A^{\otimes\infty}\to A_{[1/4,1/2]}\subseteq A$.
		Composing $\iota$ with the Jiang--Su embedding proves the assertion.
	\end{proof}
	
	\subsection{Endpoint absorption}
	
	The next lemma is a routine consequence of the equivariant
	$\cZ$-stability of the Bernoulli shift on $\bigotimes_{\Zz}\cZ$
	\cite[Theorem~B]{GL}.
	\begin{lemma}\label{lem:shiftblocks}
		Let
		\[
		\mathcal N=\bigotimes_{k\in\Zz}\cZ,
		\]
		and let $\sigma$ shift coordinate $k$ to coordinate $k+1$.
		Identify $\mathcal N_{[-m,m]}=\bigotimes_{k=-m}^{m}\cZ$ with its
		unital image in $\mathcal N$. There exist integers $m_n\geq0$
		and unital completely positive maps
		\[
		\psi_n:\cZ\longrightarrow\mathcal N_{[-m_n,m_n]}
		\]
		such that, for every $z,w\in\cZ$ and $r\in\Zz$,
		\[
		\|\psi_n(zw)-\psi_n(z)\psi_n(w)\|\longrightarrow0,
		\qquad
		\|\sigma^r(\psi_n(z))-\psi_n(z)\|\longrightarrow0.
		\]
	\end{lemma}
	
	\begin{proof}
		Since $\Zz$ is amenable and $\cZ$ is a finite strongly self-absorbing
		$C^*$-algebra \cite[Example 1.14(iv)]{TomsWinter}, the Bernoulli shift
		$\sigma$ is cocycle conjugate to $\sigma\otimes\id_{\cZ}$ by
		\cite[Theorem~B]{GL}. (As explained in \cite{GL}, the implication used
		here follows from work of Gardella--Hirshberg--Vaccaro, see \cite{GHV}, and of
		Sato \cite{SatoAmenable}. Since $\cN\cong\cZ$
		\cite[Corollary 1.11]{TomsWinter} has a unique tracial state, this
		cocycle conjugacy is also a special case of \cite[Theorem~8.5]{GHV} and
		\cite[Theorem~C]{SzaboSI}.)
		
		Let $\mathcal N_\infty=\ell^\infty(\Nz,\mathcal N)/c_0(\Nz,\cN)$,
		and let $\sigma_\infty$ be the automorphism induced by $\sigma$.
		Since $\mathcal N$ is separable, Szab\'o's equivariant McDuff-type theorem
		\cite[Corollary~3.8, (ii) $\Rightarrow$ (iii)]{SzaboSSA} gives a unital $*$-homomorphism
		$\pi:\cZ\to(\mathcal N_\infty\cap\mathcal N')^{\sigma_\infty}$; for
		unital $\mathcal N$ and the discrete group $\Zz$, the algebra
		$F_{\infty,\sigma}(\mathcal N)$ used there is
		$\mathcal N_\infty\cap\mathcal N'$. Only the $\sigma_\infty$-invariance
		of the range of $\pi$ is used below. As $\cZ$ is separable and
		nuclear, the Choi--Effros lifting theorem \cite{ChoiEffros} lifts $\pi$
		to unital completely positive maps $v_n:\cZ\to\mathcal N$, so that
		$\pi(z)=[(v_n(z))_n]$. Thus, for all $z,w\in\cZ$,
		\[
		\|v_n(zw)-v_n(z)v_n(w)\|\longrightarrow0,\qquad
		\|\sigma(v_n(z))-v_n(z)\|\longrightarrow0.
		\]
		
		Let $E_m:\mathcal N\to\mathcal N_{[-m,m]}$ be the unital completely
		positive conditional expectation that applies the tracial state of
		$\cZ$ to the factors outside $[-m,m]$; then $E_m(x)\to x$ for every
		$x\in\mathcal N$. Since $\cZ$ is separable and all maps involved are
		contractive, we may choose $m_n$ so that $\psi_n=E_{m_n}v_n$ satisfies
		$\epsilon_n(z)=\|\psi_n(z)-v_n(z)\|\to0$ for every $z\in\cZ$.
		Replacing $v_n$ by $\psi_n$ changes the two quantities displayed above
		by at most $\epsilon_n(zw)+\|w\|\epsilon_n(z)+\|z\|\epsilon_n(w)$ and
		$2\epsilon_n(z)$, respectively. Finally,
		$\|\sigma^r(x)-x\|\leq|r|\,\|\sigma(x)-x\|$ for all $r\in\Zz$ and
		$x\in\mathcal N$, since $\sigma$ is isometric.
	\end{proof}
	
	The finite-block maps can now be placed in interval subgroups approaching
	the left endpoint.
	\begin{proposition}\label{prop:insertion}
		If there is a unital homomorphism $\cZ\to C_\lambda^*(F)$, then
		$C_\lambda^*(F)\cong C_\lambda^*(F)\otimes\cZ$.
	\end{proposition}
	
	\begin{proof}
    We write $A=C_\lambda^*(F)$ as in the preliminaries.
		We construct a unital homomorphism from $\cZ$ into
		$A_\infty\cap A'$, using the notation of
		Subsection~\ref{subsec:Z-stability}.
		Since $\cZ$ is simple and $A\ne0$, the given unital homomorphism
		is injective. Compose it with the embedding
		$A\to A_{I_0}$ induced by affine rescaling, where $I_0=[1/4,1/2]$.
		This gives $\zeta_0:\cZ\to A_{I_0}$.
		Using $s$ from \eqref{eq:s-generator}, put $I_j=2^{-j}I_0$ and
		$\zeta_j=\Ad(\lambda_s^j)\zeta_0$ for $j\geq0$.
		The intervals have pairwise disjoint interiors, and
		$\zeta_j(\cZ)\subseteq A_{I_j}$. Restricting the disjoint-interval
		tensor embeddings to these subalgebras gives isometric embeddings of
		every finite tensor product of the copies of $\cZ$. Moreover,
		\begin{equation}\label{eq:intervalshift}
			\Ad(\lambda_s)\zeta_j=\zeta_{j+1}.
		\end{equation}
		
		Choose $m_n$ and $\psi_n$ as in Lemma~\ref{lem:shiftblocks}, and
		set $L_n=m_n+2n$. Embed the expanded window
		\[
		\Theta_n:\mathcal N_{[-m_n-n,m_n+n]}\longrightarrow A
		\]
		by placing coordinate $k$ in $\zeta_{L_n+k}(\cZ)$.
		Every interval index is at least $n$, and the range of $\psi_n$
		lies in the smaller window $\mathcal N_{[-m_n,m_n]}$. The maps
		$\rho_n=\Theta_n\psi_n:\cZ\to A$ are unital completely positive
		and asymptotically multiplicative by Lemma~\ref{lem:shiftblocks}.
		
		For every fixed $g\in F$, finite piecewise linearity gives
		$r(g)\in\Zz$ and $\varepsilon_g>0$ with
		$g(t)=2^{r(g)}t$ on $[0,\varepsilon_g]$.
		For all sufficiently large interval indices $j$, the maps $g$
		and $s^{-r(g)}$ agree on $I_j$. The conjugate of an element
		of $F_{I_j}$ is supported in their common image interval. On that
		interval the two conjugations agree, and outside it both are the
		identity. Consequently,
		\[
		\Ad(\lambda_g)\zeta_j(z)=\zeta_{j-r(g)}(z).
		\]
		For $n>|r(g)|$ sufficiently large, the smaller window
		$\mathcal N_{[-m_n,m_n]}$ and its shift by $-r(g)$ lie in the
		expanded window, all relevant intervals are in the linearity region,
		and all indices are nonnegative. Hence isometry gives
		\begin{equation}\label{eq:endpointcentral}
			\|[\lambda_g,\rho_n(z)]\|
			=\|\sigma^{-r(g)}(\psi_n(z))-\psi_n(z)\|
			\longrightarrow0
		\end{equation}
		by Lemma~\ref{lem:shiftblocks}.
		For a finitely supported element, sum these estimates over its finite support.
		For arbitrary $a\in A$, approximate by such a polynomial $p$
		and use $\|[a-p,\rho_n(z)]\|\leq2\|a-p\|\|z\|$.
		Thus $\|[a,\rho_n(z)]\|\to0$ for every $a\in A$ and $z\in\cZ$.
		
		Asymptotic multiplicativity of the unital completely positive maps $\rho_n$
		implies that $z\mapsto[(\rho_n(z))_n]$ is a unital homomorphism into
		$A_\infty$, and the commutator estimates place its range in
		$A_\infty\cap A'$. Toms--Winter's absorption criterion
		\cite[Theorem 2.2]{TomsWinter} now gives
		$A\cong A\otimes_{\min}\cZ$.
	\end{proof}
	
	\begin{proof}[Proof of Theorem~\ref{thm:Z-stability}]
		Proposition~\ref{prop:localZ} supplies a unital copy of $\cZ$ in $C_\lambda^*(F)$.
		Proposition~\ref{prop:insertion} then gives
		$C_\lambda^*(F)\cong C_\lambda^*(F)\otimes_{\min}\cZ$.
	\end{proof}

	\section{Quasitrace rigidity and the proof of selflessness}
	
	We prove uniqueness of the normalised quasitrace on $C_\lambda^*(F)$ and combine it
	with Theorem~\ref{thm:Z-stability} to obtain strict comparison and
	selflessness.
	
	\subsection{Quasitraces on the commutator subgroup}\label{sec:restriction}
	
	The following lemma shows how commuting conjugates force quasitraces
	to be additive. For $F'$, the required disjoint-support
	displacement appears in \cite[proof of Proposition~2.1]{Jolissaint2005};
	related generator shifts and their local inner implementation are
	described in \cite[Lemma~4.1 and Remark~4.3(i)]{HaagerupPicioroaga}.
	
	\begin{lemma}\label{lem:commutingconjugates}
		Let $G$ be a group such that for every finite set $S\subseteq G$ there exists $g\in G$ such that every element of $S$ commutes with every element of $gSg^{-1}$. Let $D$ be a unital $C^*$-algebra and $\pi\colon G\to\mathcal U(D)$ a homomorphism into the unitary group of $D$ whose range has dense linear span. Then, every quasitrace $q\in\QT_1(D)$ is a tracial state on $D$.
	\end{lemma}
	\begin{proof}
		For a subgroup $H\leq G$, let $D_H$ be the $C^*$-subalgebra of $D$
		generated by $\pi(H)$.
		
		\emph{Step 1: pairwise commuting conjugates.} Let $S\subseteq G$ be
		finite. We show by induction that for every $k\geq0$ there are
		$h_1,\ldots,h_{2^k}\in G$ such that, whenever $i\neq j$, every element
		of $h_iSh_i^{-1}$ commutes with every element of $h_jSh_j^{-1}$. For
		$k=0$, take $h_1=1$. Given $h_1,\ldots,h_N$ with $N=2^k$, apply the
		hypothesis to the finite set $S'=\bigcup_{i=1}^Nh_iSh_i^{-1}$. This gives
		$g\in G$ such that every element of $S'$ commutes with every element of
		$gS'g^{-1}$. Put $h_{N+i}=gh_i$ for $1\leq i\leq N$. Pairs of distinct
		indices in $\{1,\ldots,N\}$ are covered by the inductive hypothesis, and
		pairs of distinct indices in $\{N+1,\ldots,2N\}$ by its conjugate under
		$g$. If $i\leq N<j$, then $h_iSh_i^{-1}\subseteq S'$ and
		$h_jSh_j^{-1}\subseteq gS'g^{-1}$, which commute elementwise by the
		choice of $g$.
		
		Let $H=\langle S\rangle$, $H_i=h_iHh_i^{-1}$ and $u_i=\pi(h_i)$.
		Centralisers are subgroups, so $H_i$ and $H_j$ commute elementwise for
		$i\neq j$. Since $\pi(H_i)$ and $\pi(H_j)$ commute elementwise and
		are closed under adjoints, the $C^*$-subalgebras
		$D_{H_i}=u_iD_Hu_i^*$, $1\leq i\leq2^k$, commute pairwise.
		
		\emph{Step 2: additivity on $D_{\langle S\rangle}$.} Fix
		$q\in\QT_1(D)$, a finite set $S\subseteq G$, and selfadjoint
		contractions $x,y\in D_H$, where $H=\langle S\rangle$. For $k\geq0$, let
		$N=2^k$, choose $h_1,\ldots,h_N$ as in Step~1, and set
		\[
		x_i=u_ixu_i^*,\qquad y_i=u_iyu_i^*,\qquad
		X_k=\frac1N\sum_{i=1}^Nx_i,\qquad Y_k=\frac1N\sum_{i=1}^Ny_i.
		\]
		For $i\neq j$, both $x_i$ and $y_i$ commute with both $x_j$ and
		$y_j$. In particular, $x_1,\ldots,x_N$ are pairwise commuting
		selfadjoint elements, so they lie in an abelian $C^*$-subalgebra of
		$D$, on which $q$ is linear. By unitary invariance, recalled in
		Subsection~\ref{subsec:quasitraces}, $q(X_k)=q(x)$.
		In the same way $q(Y_k)=q(y)$, and the pairwise
		commuting elements $x_i+y_i=u_i(x+y)u_i^*$ give $q(X_k+Y_k)=q(x+y)$.
		In the commutator, only terms with equal indices survive:
		\[
		[X_k,Y_k]=\frac1{N^2}\sum_{i,j=1}^N[x_i,y_j]
		=\frac1{N^2}\sum_{i=1}^Nu_i[x,y]u_i^*,
		\qquad\text{so}\qquad
		\|[X_k,Y_k]\|\leq\frac{\|[x,y]\|}{2^k}.
		\]
		The elements $X_k,Y_k$ are selfadjoint contractions and their commutator
		norms tend to zero. Thus \eqref{eq:almostcommuting} gives
		\[
		|q(x+y)-q(x)-q(y)|
		=|q(X_k+Y_k)-q(X_k)-q(Y_k)|\longrightarrow0.
		\]
		Hence $q(x+y)=q(x)+q(y)$. Since $q$ is homogeneous on each
		$C^*(1,z)$ with $z$ selfadjoint, this identity extends to all
		selfadjoint $x,y\in D_H$.
		
		\emph{Step 3: additivity on $D$.} If $S\subseteq S'$, then
		$D_{\langle S\rangle}\subseteq D_{\langle S'\rangle}$. Hence the union
		$D_0$ of the algebras $D_{\langle S\rangle}$, over all finite
		$S\subseteq G$, is directed and is a $*$-subalgebra of $D$. It
		contains $\pi(G)$, so it is dense in $D$, and real parts of
		approximants show that $(D_0)_{\mathrm{sa}}$ is dense in
		$D_{\mathrm{sa}}$. Any two selfadjoint elements of $D_0$ lie in a common
		$D_{\langle S\rangle}$, so $q$ is additive on $(D_0)_{\mathrm{sa}}$ by
		Step~2. Norm continuity from Subsection~\ref{subsec:quasitraces} then
		extends additivity to $D_{\mathrm{sa}}$.
		
		\emph{Step 4: conclusion.}
        Homogeneity on selfadjoint elements and (Q2) show that $q$ is complex linear. By (Q1) and $q(1)=1$, it is a state. Finally, (Q1) and polarization imply $q(xy)=q(yx)$ for all $x,y\in D$, so $q$ is tracial.
	\end{proof}
	
	\begin{proposition}\label{prop:restriction}
		Every normalised quasitrace on $C_\lambda^*(F')$ is a trace.
	\end{proposition}
	\begin{proof}
		Let $S\subseteq F'$ be finite. By Lemma~\ref{lem:F'}, every element
		of $S$ is the identity near both endpoints, so their supports lie in
		a nondegenerate closed dyadic interval $K\subseteq(0,1)$. Choose a
		nondegenerate closed dyadic interval $J\subseteq(0,1)$ disjoint from
		$K$. Lemma~\ref{lem:interval-moving} gives $g\in F'$ with $g(K)=J$,
		so every element of $gSg^{-1}$ is supported in $J$. A homeomorphism
		supported in a closed interval maps that interval onto itself. Hence
		homeomorphisms supported in disjoint closed intervals commute, and
		every element of $S$ commutes with every element of $gSg^{-1}$.
		
		Lemma~\ref{lem:commutingconjugates}, applied to
		$F'\to\mathcal U(C_\lambda^*(F'))$, $h\mapsto \lambda_h$, shows that every $q\in\QT_1(C_\lambda^*(F'))$ is a tracial state.
	\end{proof}
	
	Consequently, every normalised quasitrace on $C_\lambda^*(F)$ restricts to
	$\tau_{F'}$ on $C_\lambda^*(F')$, by the uniqueness of the tracial state on $C_\lambda^*(F')$
	recorded after Theorem~\ref{thm:LBMB}. It remains to show that this
	restriction determines the quasitrace on all of $C_\lambda^*(F)$.
	
	\subsection{Endpoint shifts and continuous towers}\label{sec:towers}
	
	We construct the endpoint shift algebras and continuous Rokhlin towers
	used in Subsection~\ref{sec:compression}. Endpoint localisation
	provides exact commutation with prescribed coefficients in $\Cz[F']$.
	
	Set $\rho_k=\Ad(W_k)|_{C_\lambda^*(F')}$ for $k\in\Zz^2$. Since $k\mapsto W_k$
	is a unitary representation, these automorphisms form an action of
	$\Zz^2$ and preserve $\tau_{F'}$. For $g\in F$, let
	$\operatorname{supp}(g)$ denote the closure of $\{t:g(t)\ne t\}$.
	
	\subsubsection*{Endpoint shifts}
	
	\begin{lemma}\label{lem:endpoint-bumps}
		There exist nonidentity elements $z,w\in F'$ such that, on setting
		\[
		z_j=u^jzu^{-j},\qquad w_j=v^jwv^{-j}
		\qquad(j\in\Zz),
		\]
		all the supports of the $z_j$ and $w_j$ are pairwise disjoint, and
		\[
		H:=\langle z_j,w_j:j\in\Zz\rangle
		\cong
		\left(\bigoplus_{j\in\Zz}\Zz\right)
		\oplus
		\left(\bigoplus_{j\in\Zz}\Zz\right),
		\]
		with the displayed elements as the standard generators. Moreover,
		\[
		uz_ju^{-1}=z_{j+1},\qquad uw_ju^{-1}=w_j,
		\qquad
		vz_jv^{-1}=z_j,\qquad vw_jv^{-1}=w_{j+1}.
		\]
		For every $0<\varepsilon<1/2$, there exists $N$ such that
		\[
		\operatorname{supp}(z_j)\subseteq(0,\varepsilon),\qquad
		\operatorname{supp}(w_j)\subseteq(1-\varepsilon,1)
		\qquad(j\geq N).
		\]
	\end{lemma}
	
	\begin{proof}
		We use the disjoint-conjugate construction of
		\cite[Lemma~2.8]{Brin1999Ubiquity}. Choose nonidentity $z,w\in F$
		supported in $[5/32,6/32]$ and $[26/32,27/32]$, respectively;
		these exist by \cite[Lemma~4.4]{CFP}.
		
		The definitions of \(u\) and \(v\) give
		\[
		u(1/4)=1/8,\qquad v(3/4)=7/8,
		\]
		and \(u(t)<t\) on \((0,1/2)\), while \(v(t)>t\) on \((1/2,1)\).
		The intervals
		\[
		\bigl(u^{j+1}(1/4),u^j(1/4)\bigr),\qquad j\in\Zz,
		\]
		are pairwise disjoint and contain the respective supports of \(z_j\).
		Likewise, the pairwise disjoint intervals
		\[
		\bigl(v^j(3/4),v^{j+1}(3/4)\bigr),\qquad j\in\Zz,
		\]
		contain the respective supports of \(w_j\). The two families lie in
		different halves of \([0,1]\). Each conjugate is the identity near both
		endpoints, hence belongs to \(F'\) by \cite[Theorem~4.1]{CFP}.
		
		A nonidentity increasing interval homeomorphism has infinite order: the iterates of a point that it moves form a strictly monotone sequence.
		Thus every \(z_j\) and \(w_j\) has infinite order. Disjoint supports
		give commutation, and restricting a finite product to each supporting
		interval shows that it is the identity only when every exponent is zero.
		This proves the asserted direct-sum decomposition. The conjugation
		relations follow from the definitions and the separation of the two halves.
		
		Finally, \(u^j(1/4)\) decreases to a fixed point of \(u\), necessarily
		zero, and \(v^j(3/4)\) increases to a fixed point of \(v\), necessarily
		one. The support inclusions above give the final assertion.
	\end{proof}
	
	\begin{lemma}\label{lem:endpoint-shift-algebras}
		Let \(X=\Tz^{\Zz}\), let \(q_j\in C(X)\) be the coordinate unitary
		\(q_j(x)=x_j\), and define \(\sigma\in\operatorname{Aut}(C(X))\) by
		\(\sigma(q_j)=q_{j+1}\). Let \(\mu\) be product Haar measure on \(X\).
		For the elements of Lemma~\ref{lem:endpoint-bumps}, there is a unital
		injective \(*\)-homomorphism
		\[
		\Theta:C(X)\otimes C(X)\longrightarrow C_\lambda^*(F')
		\]
		satisfying
		\[
		\Theta(q_j\otimes1)=\lambda_{z_j},\qquad
		\Theta(1\otimes q_j)=\lambda_{w_j}.
		\]
		It is equivariant and trace preserving, in the sense that
		\[
		\rho_{(m,n)}\circ\Theta
		=\Theta\circ(\sigma^m\otimes\sigma^n)
		\qquad((m,n)\in\Zz^2)
		\]
		and
		\[
		\tau_{F'}(\Theta(f\otimes g))
		=\left(\int_X f\,d\mu\right)
		\left(\int_X g\,d\mu\right)
		\qquad(f,g\in C(X)).
		\]
	\end{lemma}
	
	\begin{proof}
		By the subgroup inclusion recorded in Section~2.1, the subgroup
		$H$ of Lemma~\ref{lem:endpoint-bumps} gives an isometric unital
		embedding $C_\lambda^*(H)\subseteq C_\lambda^*(F')$.
		
		For finite sets \(Q,R\subseteq\Zz\), the corresponding generators form
		a free abelian subgroup of rank \(|Q|+|R|\). Fourier transformation
		identifies its reduced group \(C^*\)-algebra with
		\(C(\Tz^Q\times\Tz^R)\), sending the group generators to coordinate
		unitaries. These identifications are compatible as \(Q,R\) increase.
		Since finite-coordinate functions are uniformly dense in \(C(X\times X)\),
		they give the asserted embedding \(\Theta\).
		
		On every finite-coordinate algebra, the canonical group trace is
		integration against Haar measure: both vanish on every nonconstant
		character. Density therefore gives the displayed trace formula.
		The conjugation relations of Lemma~\ref{lem:endpoint-bumps} establish
		equivariance on the coordinate unitaries, which generate the domain.
	\end{proof}
	
	\begin{lemma}\label{lem:endpoint-localisation}
		For finite \(Q\subseteq\Zz\), write
		\[
		D_Q=C^*(1,q_j:j\in Q)\subseteq C(X).
		\]
		Given a finite set \(\mathcal E\subseteq\Cz[F']\) and finite sets
		\(Q,R\subseteq\Zz\), there exists \(N\geq0\) such that
		\[
		\Theta\bigl(D_{Q+m}\otimes_{\min}D_{R+n}\bigr)
		\subseteq C_\lambda^*(F')\cap C^*(\mathcal E)'
		\qquad(m,n\in\Zz,\ m,n\geq N),
		\]
		where \(Q+m=\{j+m:j\in Q\}\), and similarly for \(R+n\).
	\end{lemma}
	
	\begin{proof}
		We use the endpoint-support argument from
		\cite[proof of Corollary~2.5]{Jolissaint}.
		Only finitely many group elements occur in the polynomials in
		\(\mathcal E\). By \cite[Theorem~4.1]{CFP}, their supports are contained
		in \([\varepsilon,1-\varepsilon]\) for some \(0<\varepsilon<1/2\).
		Lemma~\ref{lem:endpoint-bumps} and the finiteness of \(Q,R\) give
		\(N\) such that, for all \(m,n\geq N\),
		\[
		\operatorname{supp}(z_{j+m})\subseteq(0,\varepsilon)
		\quad(j\in Q),\qquad
		\operatorname{supp}(w_{j+n})\subseteq(1-\varepsilon,1)
		\quad(j\in R).
		\]
		These elements commute with every group element occurring in
		\(\mathcal E\). Hence their canonical unitaries, and their adjoints,
		commute with every element of \(\mathcal E\) and \(\mathcal E^*\).
		Taking generated \(C^*\)-algebras gives the required inclusion.
	\end{proof}
	
	\subsubsection*{Continuous towers}
	
	We use a continuous refinement of the classical Rokhlin lemma. Here finite-coordinate dependence permits endpoint
	localisation, while exact orthogonality removes the residual quotient
	terms in Lemma~\ref{lem:compression}. We include the short regularity
	argument.
	\begin{lemma}\label{lem:tower}
		Let \(\sigma\) be the bilateral coordinate-shift automorphism of
		\(C(\Tz^\Zz)\), and let \(\mu\) be product Haar measure. For every
		integer \(L\geq1\) and \(\eta>0\), there exists a continuous positive
		contraction \(f\in C(\Tz^\Zz)\), depending on finitely many coordinates,
		such that
		\[
		f\sigma^j(f)=0\quad(0<|j|\leq L-1),\qquad
		L\mu(f^2)>1-\eta.
		\]
	\end{lemma}
	
	\begin{proof}
		For $L=1$, take $f=1$. By decreasing $\eta$ if necessary, we may
		otherwise suppose $0<\eta<1$.
		Put \(X=\Tz^\Zz\), and let \(S\colon X\to X\) be the coordinate
		shift determined by \(\sigma(g)=g\circ S\).
		For every nonzero \(p\in\Zz\), the set \(\{x:S^px=x\}\) lies
		in the null set \(\{x:x_0=x_p\}\). Thus \(S\) is an aperiodic
		automorphism of the atomless standard probability space \((X,\mu)\).
		Rokhlin's lemma, with error \(\eta/2\), gives
		a measurable set \(E\) such that \(E,S^{-1}E,\ldots,S^{-(L-1)}E\)
		are pairwise disjoint and \(L\mu(E)\geq1-\eta/2>1-\eta\).
		
		By inner regularity choose a compact \(K\subseteq E\) with
		\(L\mu(K)>1-\eta\). Choose pairwise disjoint open neighbourhoods
		\(U_j\supset S^{-j}K\), \(0\leq j<L\), and put
		\(O=\bigcap_{j=0}^{L-1}S^jU_j\). Then \(K\subseteq O\) and
		\(S^{-j}O\subseteq U_j\) for \(0\leq j<L\), so
		\(O\cap S^{-j}O=\varnothing\) for \(1\leq j<L\).
		Applying shifts gives the same disjointness for negative \(j\).
		
		Basic cylinders cover \(K\) inside \(O\). Taking a finite subcover
		and the union of its coordinate sets gives a finite \(Q\subset\Zz\)
		and an open \(V\subseteq\Tz^Q\) with
		\(K\subseteq\pi_Q^{-1}(V)\subseteq O\), where \(\pi_Q\) is the
		coordinate projection. Choose a continuous
		\(\varphi\colon\Tz^Q\to[0,1]\) equal to one on \(\pi_Q(K)\)
		and zero off \(V\), and put \(f=\varphi\circ\pi_Q\).
		Then \(f\) depends on finitely many coordinates and vanishes off
		\(O\), giving \(f\sigma^j(f)=0\) for \(0<|j|<L\). Since
		\(f=1\) on \(K\), we have \(L\mu(f^2)\geq L\mu(K)>1-\eta\).
	\end{proof}
	
	\begin{remark}
		The mass estimate uses product Haar measure. For $L\geq2$,
		orthogonality forces $f$ to vanish at every fixed point of the shift,
		so the estimate fails for invariant point masses when $0<\eta<1$.
	\end{remark}
	
	\subsection{Exact matrix compression}\label{sec:compression}
	
	We combine the endpoint towers with matrix compression to obtain an exact factorisation over $C_\lambda^*(F')$ with arbitrarily small canonical trace deficit. This is analogous to the tracial Rokhlin approximation method in \cite{OsakaPhillips}. We first record the positive
coefficient matrix used in this construction. We first record the positive coefficient matrix used in this construction.
	
	Let $E:C_\lambda^*(F)\to C_\lambda^*(F')$ be the canonical conditional expectation determined on
	generators by
	\begin{equation}
		E(\lambda_g)=\begin{cases}
			\lambda_g & \text{ if } g\in F',\\
			0 & \text{ if } g\not\in F'.
		\end{cases}
	\end{equation}
	Thus $\tau_{F'}\circ E=\tau_F$.
	For a finite set $\Lambda\subseteq\Zz^2$, put
	\begin{equation}\label{eq:phi}
		\Phi_\Lambda(a)\coloneqq[E(W_p^*aW_q)]_{p,q\in\Lambda}.
	\end{equation}
	This map is completely positive: $[W_p^*aW_q]$ is obtained from $a$
	by a row-operator compression, followed by the matrix amplification
	of $E$. Also $\Phi_\Lambda(1)=I$, since distinct $W_p$ have distinct
	endpoint quotient classes. For $a$ as in \eqref{eq:fourier},
	\begin{equation}\label{eq:twistedmatrix}
		\Phi_\Lambda(a)_{pq}=\rho_{-p}(b_{p-q}),
	\end{equation}
	with absent coefficients interpreted as zero. In particular,
	\(0\leq\Phi_\Lambda(a)\leq\|a\|I\) when \(a\geq0\).
	
	The coefficient-matrix map $\Phi_\Lambda$ is the usual finite-corner
	compression for reduced crossed products; see \cite{BrownOzawa}. The proof below
	combines this compression with the Rokhlin tower matrix construction in
	\cite[proof of Theorem~4.1, pp.~655--658]{HWZ}.
	Endpoint localisation provides exact commutation with the relevant
	coefficients, yielding the exact factorisation below.
	
	\begin{samepage}
		\begin{lemma}\label{lem:compression}
			For every $a\in\Cz[F]_+\coloneqq \Cz[F]\cap C_\lambda^*(F)_+$ and $0<\delta<1$, there exist a positive contraction $c\in C_\lambda^*(F')$, an integer $d\geq1$, and a row $y\in M_{1,d}(C_\lambda^*(F))$ such that
			\begin{equation}
				\tau_F(1-c^2)<\delta,\qquad yy^*=cac,\qquad y^*y\in M_d(C_\lambda^*(F'))_+.
			\end{equation}
		\end{lemma}
	\end{samepage}
	
	\begin{proof}
		\emph{Step 1: a cutoff of large trace.}
		Choose an integer \(R\geq0\) so that the nonzero coefficients in
		\eqref{eq:fourier} have \(\|k\|_\infty\leq R\). Put
		\[
		m=\max\{2,\lceil8R/\delta\rceil\},\quad L=m+R,\quad
		\eta=\delta/8,\quad \Lambda=\{0,\ldots,m-1\}^2,\quad d=m^2,
		\]
		and let
		\[
		B_0=C^*(1,\rho_{-p}(b_k):p\in\Lambda,\ \|k\|_\infty\leq R).
		\]
		Apply Lemma~\ref{lem:tower} with height $L$ and error $\eta$ to obtain
		$f,g\in C(\Tz^\Zz)$, depending on finite coordinate sets $Q,R_0$,
		respectively. Apply Lemma~\ref{lem:endpoint-localisation} to these sets
		and the finite family of coefficients generating $B_0$. For sufficiently
		large $M,N$, the positive contractions
		\[
		f_-=\Theta(\sigma^M(f)\otimes1),\qquad
		f_+=\Theta(1\otimes\sigma^N(g))
		\]
		commute with $B_0$. The shifts preserve orthogonality and Haar integrals.
		Set $h=f_-f_+$. Equivariance and the product trace formula in
		Lemma~\ref{lem:endpoint-shift-algebras} give
		\begin{equation}\label{eq:baseproperties}
			[h,B_0]=0,\quad h\rho_r(h)=0\ (0<\|r\|_\infty\leq L-1),\quad
			\tau_F(h^2)>(1-\eta)^2/L^2.
		\end{equation}
		For \(r\ne0\), a nonzero coordinate supplies the vanishing factor in
		the middle identity. Define $h_p=W_phW_p^*$ and
		$c=\sum_{p\in\Lambda}h_p$. The $h_p$ are pairwise orthogonal positive
		contractions in $C_\lambda^*(F')$, so $0\leq c\leq1$. Moreover,
		\[
		\tau_F(c^2)=m^2\tau_F(h^2)
		>\left(\frac m{m+R}\right)^2(1-\eta)^2
		\geq(1-\delta/8)^4\geq1-\delta/2>1-\delta.
		\]
		Here we used $R/m\leq\delta/8$ and $(1-t)^4\geq1-4t$ for $0\leq t\leq1$.
		
		\emph{Step 2: cancellation of the residual quotient terms.}
		The tower height $L=m+R$ accounts for both the matrix window and the
		Fourier radius. Each term in the expansion of $cac$ has the form
		\begin{equation}\label{eq:residual}
			h_p b_kW_kh_q
			=W_ph\rho_{-p}(b_k)W_{k+q-p}hW_q^*.
		\end{equation}
		Its residual exponent satisfies
		\(\|k+q-p\|_\infty\leq R+m-1=L-1\). If it is nonzero, commute
		the coefficient past \(h\) and use
		\(hW_rh=h\rho_r(h)W_r=0\). If it is zero, the term is
		\(W_ph^2\rho_{-p}(b_{p-q})W_q^*\).
		
		\emph{Step 3: the positive matrix factorisation.}
		Set \(X=\Phi_\Lambda(a)\). Its entries lie in \(B_0\), and its
		positivity in \(M_d(C_\lambda^*(F'))\) implies positivity in \(M_d(B_0)\) by
		spectral permanence. Summing the surviving terms and using
		\eqref{eq:twistedmatrix} gives
		\begin{equation}\label{eq:compressedmatrix}
			cac=\sum_{p,q\in\Lambda}W_ph^2X_{pq}W_q^*.
		\end{equation}
		Define
		\[
		V=(W_p)_{p\in\Lambda},\qquad
		H=\operatorname{diag}(h,\ldots,h),\qquad y=VHX^{1/2}.
		\]
		Here \(V,y\in M_{1,d}(C_\lambda^*(F))\) and \([H,X]=0\), since \([h,B_0]=0\).
		For \(p,q\in\Lambda\), we have \(\|q-p\|_\infty\leq m-1\leq L-1\),
		so \eqref{eq:baseproperties} implies
		\begin{equation}\label{eq:roworthogonality}
			(HV^*VH)_{pq}=hW_{q-p}h=\delta_{pq}h^2.
		\end{equation}
		Consequently,
		\begin{equation}\label{eq:rowfactorization}
			\begin{aligned}
				yy^*&=VHXHV^*=cac,\\
				y^*y&=X^{1/2}HV^*VHX^{1/2}=H^2X\in M_d(C_\lambda^*(F'))_+.
			\end{aligned}
		\end{equation}
	\end{proof}
	
	\subsection{Quasitrace rigidity}\label{sec:rigidity}
	
	We combine Lemma~\ref{lem:compression} with a small-remainder
	estimate to prove quasitrace rigidity without an exactness assumption.
	This strategy is related to the quasitrace argument in
	\cite[Proposition~4.6]{FuLin}.
	
	The following estimate is a quantitative consequence of the quasitrace
	inequalities, in particular
	\cite[Lemma 3.5(1)]{Haagerup}.
	\begin{lemma}
		\label{lem:remainder}
		Let \(q\) be a normalized quasitrace and \(t\) a tracial state on a
		unital \(C^*\)-algebra \(D\). Suppose \(a\in D_+\), \(c\in D\),
		\(0\leq c\leq1\),
		\(q(cac)=t(cac)\), and \(q(1-c^2)=t(1-c^2)=\beta\). Then,
		\[
		|q(a)-t(a)|\leq2\|a\|\sqrt\beta.
		\]
	\end{lemma}
	
	\begin{proof}
		Write \(M=\|a\|\), \(d_0=1-c^2\),
		\(b=a^{1/2}c^2a^{1/2}\), and \(r=a^{1/2}d_0a^{1/2}\).
		Then \(a=b+r\), with \(b,r\geq0\). The quasitrace identity gives
		\[
		q(b)=q(cac),\qquad
		q(r)=q(d_0^{1/2}ad_0^{1/2})\leq Mq(d_0)=M\beta.
		\]
		Order preservation and \eqref{eq:sqrt} imply
		\[
		0\leq\sqrt{q(a)}-\sqrt{q(b)}\leq\sqrt{M\beta}.
		\]
		Since \(q(a),q(b)\leq M\), multiplying by
		\(\sqrt{q(a)}+\sqrt{q(b)}\leq2\sqrt M\) gives
		\(0\leq q(a)-q(cac)\leq2M\sqrt\beta\). For the trace, additivity gives
		\(0\leq t(a)-t(cac)=t(r)\leq M\beta\). Subtracting these scalar
		differences and using \(q(cac)=t(cac)\) gives
		\(-M\beta\leq q(a)-t(a)\leq2M\sqrt\beta\).
		Since \(0\leq\beta\leq1\), the stated bound follows, including when
		\(M=0\) or \(\beta=0\).
	\end{proof}
	
	\begin{proposition}
		\label{prop:rigidity}
		If $q\in\QT_1(C_\lambda^*(F))$ satisfies $q|_{C_\lambda^*(F')}=\tau_{F'}$, then $q=\tau_F$.
	\end{proposition}
	\begin{proof}
		Fix \(a\in\Cz[F]_+\) and \(0<\delta<1\). Apply
		Lemma~\ref{lem:compression} to obtain \(c\), \(d\), and \(y\).
		Let \(Y\in M_d(C_\lambda^*(F))\) have first row \(y\) and all other rows zero.
		Then
		\[
		YY^*=\operatorname{diag}(cac,0,\ldots,0),\qquad
		Y^*Y=y^*y\in M_d(C_\lambda^*(F'))_+.
		\]
		Let $q_d$ be the compatible unnormalised matrix amplification of $q$.
		Its restriction to $M_d(C_\lambda^*(F'))$ extends $q|_{C_\lambda^*(F')}=\tau_{F'}$, so uniqueness
		of amplification, recalled in Subsection~\ref{subsec:quasitraces}, gives
		$q_d|_{M_d(C_\lambda^*(F'))}=\tau_{F'}\otimes\Tr$. Hence
		\[
		\begin{aligned}
			q(cac)&=q_d(YY^*)=q_d(Y^*Y)\\
			&=(\tau_{F'}\otimes\Tr)(Y^*Y)
			=(\tau_F\otimes\Tr)(YY^*)=\tau_F(cac).
		\end{aligned}
		\]
		Also \(1-c^2\in C_\lambda^*(F')\), so
		\[
		q(1-c^2)=\tau_F(1-c^2)<\delta.
		\]
		Lemma~\ref{lem:remainder} yields
		$|q(a)-\tau_F(a)|\leq2\|a\|\sqrt\delta$.
		Letting $\delta\downarrow0$ gives $q(a)=\tau_F(a)$ for every
		\(a\in\Cz[F]_+\).
		
		For \(x\in C_\lambda^*(F)_+\), approximate \(x^{1/2}\) in norm by finitely supported elements
		\(s_j\). Then \(s_j^*s_j\to x\), so norm continuity extends the equality
		to \(C_\lambda^*(F)_+\). Positive and negative parts, followed by real and imaginary
		parts, extend it to all of \(C_\lambda^*(F)\).
	\end{proof}
	
	\begin{proposition}\label{prop:uniqueQT}
		The tracial state $\tau_F$ is the unique normalised quasitrace on $C_\lambda^*(F)$.
	\end{proposition}
	
	\begin{proof}
		Let $q\in\QT_1(C_\lambda^*(F))$. Proposition~\ref{prop:restriction} makes $q|_{C_\lambda^*(F')}$
		a tracial state, so $q|_{C_\lambda^*(F')}=\tau_{F'}$ by the uniqueness recorded after
		Theorem~\ref{thm:LBMB}. Proposition~\ref{prop:rigidity} now gives
		$q=\tau_F$.
	\end{proof}

	\subsection{Proof of selflessness}
	
	\begin{proof}[Proof of Theorem~\ref{thm:selfless}]
		The algebra $C_\lambda^*(F)$ is unital, simple and has a unique faithful tracial state $\tau_F$ by
		Theorem~\ref{thm:LBMB} (Le Boudec--Matte Bon) and the observation following it. 
		By Theorem~\ref{thm:Z-stability} and
		\cite[Proposition~5.2]{APTV}, $C_\lambda^*(F)$ has strict comparison of positive
		elements by lower semicontinuous $2$-quasitraces. The underlying
		almost-unperforation result is due to R\o rdam
		\cite[Theorem~4.5]{Rordam}. Since $C_\lambda^*(F)$ is simple and unital, it
		suffices to test this comparison on normalised quasitraces
		\cite[Paragraph~3.4]{APTV}; the correspondence with functionals on
		the Cuntz semigroup is given in \cite[Proposition~4.2]{ERS}.
		Proposition~\ref{prop:uniqueQT} therefore gives strict comparison
		by $\tau_F$ in the sense of
		Subsection~\ref{subsec:comparison-selflessness}.
		Robert's theorem \cite[Theorem~1.1]{RobertComparison}
		therefore applies and shows that $(C_\lambda^*(F),\tau_F)$ is selfless.
	\end{proof}


\begin{thebibliography}{99}

		\bibitem{AGKP25} T. Amrutam, D. Gao, S. Kunnawalkam Elayavalli and G. Patchell,
		\emph{Strict comparison in reduced group $C^*$-algebras},
		Invent. Math. \textbf{242} (2025), no.~3, 639--657.
		
		\bibitem{APTV}
		R. Antoine, F. Perera, H. Thiel, and E. Vilalta,
		\emph{Pure \(C^*\)-algebras},
		arXiv:2406.11052v3, \url{https://arxiv.org/abs/2406.11052v3}.
		
		\bibitem{Arimoto}
		R. Arimoto,
		\emph{\(C^*\)-selflessness of groups of homeomorphisms},
		arXiv:2609.29399v2, \url{https://arxiv.org/abs/2609.29399v2}.

        \bibitem{AFS26}
        G. Arzhantseva and M. Finn-Sell,
        \emph{Lacunary hyperbolic groups with fast injectivity radius growth and enough loxodromic elements are selfless},
        arXiv:2606.20456, \url{https://arxiv.org/abs/2606.20456}.
		
		\bibitem{BF26} A. Basu and F. Flores,
		\emph{Selfless inclusions arising from commensurator groups of hyperbolic groups},
		arXiv:2605.12737v1 (2026), \url{https://arxiv.org/abs/2605.12737v1}.
		
		\bibitem{Belk}
		J. M. Belk,
		\emph{Thompson's group \(F\)},
		Ph.D. thesis, Cornell University, 2004,
		\href{https://arxiv.org/abs/0708.3609}{arXiv:0708.3609}.
		
		\bibitem{BS16} R. Bieri and R. Strebel, \emph{On groups of PL-homeomorphisms of the real line}. Mathematical Surveys and Monographs, 215. American Mathematical Society, Providence, RI, 2016.
		
		\bibitem{BH} B. Blackadar and D. Handelman,
		\emph{Dimension functions and traces on \(C^*\)-algebras},
		J. Funct. Anal. \textbf{45} (1982), no.~3, 297--340.
		
		\bibitem{BleakJuschenko} C. Bleak and K. Juschenko,
		\emph{Ideal structure of the C*-algebra of R. Thompson's group $T$},
		in \emph{Topological Methods in Group Theory},
		London Math. Soc. Lecture Note Ser. \textbf{451},
		Cambridge University Press, Cambridge, 2018, 46--66.
		
		\bibitem{BKKO} E. Breuillard, M. Kalantar, M. Kennedy and N. Ozawa,
		\emph{\(C^*\)-simplicity and the unique trace property for discrete groups},
		Publ. Math. Inst. Hautes \'Etudes Sci. \textbf{126} (2017), 35--71.
		
		\bibitem{Brin1999Ubiquity} M. G. Brin,
		\emph{The ubiquity of Thompson's group~$F$ in groups of piecewise
			linear homeomorphisms of the unit interval},
		J. London Math. Soc. (2) \textbf{60} (1999), no.~2, 449--460.
		
		\bibitem{BBLS} K. A. Brix, C. Bruce, K. Li and E. Scarparo,
		\emph{Maximal ideals of reduced group C*-algebras and Thompson's groups},
		Trans. Amer. Math. Soc. \textbf{379} (2026), no.~6, 4307--4321.
		
		\bibitem{BrownOzawa} N. P. Brown and N. Ozawa,
		\emph{$C^*$-algebras and finite-dimensional approximations},
		Graduate Studies in Mathematics \textbf{88}, American Mathematical Society,
		Providence, RI, 2008.
		
		\bibitem{CFP} J. W. Cannon, W. J. Floyd and W. R. Parry,
		\emph{Introductory notes on Richard Thompson's groups},
		Enseign. Math. (2) \textbf{42} (1996), no.~3--4, 215--256.

        \bibitem{CJM26} M. Chaudkhari, K. Juschenko and K. Muliarchyk,
        \emph{Thompson's groups are sofic},
        arXiv:2610.25415v1 (2026),
        \url{https://arxiv.org/abs/2610.25415v1}.
		
		\bibitem{ChoiEffros} M.-D. Choi and E. G. Effros,
		\emph{The completely positive lifting problem for \(C^*\)-algebras},
		Ann. of Math. (2) \textbf{104} (1976), no.~3, 585--609.

        \bibitem{Connes} A. Connes,
        \emph{Classification of injective factors. Cases $\mathrm{II}_1$, $\mathrm{II}_\infty$, $\mathrm{III}_\lambda$, $\lambda\ne1$}, Ann. of Math. (2) \textbf{104} (1976), no.~1, 73--115.
		
		\bibitem{DadarlatToms} M. Dadarlat and A. S. Toms,
		\emph{$\mathcal Z$-stability and infinite tensor powers of $C^*$-algebras},
		Adv. Math. \textbf{220} (2009), no.~2, 341--366.
		
		\bibitem{DudkoMedynets} A. Dudko and K. Medynets,
		\emph{Finite factor representations of Higman--Thompson groups},
		Groups Geom. Dyn. \textbf{8} (2014), no.~2, 375--389.

		
        \bibitem{Effros} E. G. Effros,
		\emph{Property {$\Gamma $} and inner amenability},
		Proc. Amer. Math. Soc. \textbf{47} (1975), 483--486.
        

		\bibitem{ERS} G. A. Elliott, L. Robert and L. Santiago,
		\emph{The cone of lower semicontinuous traces on a \(C^*\)-algebra},
		Amer. J. Math. \textbf{133} (2011), no.~4, 969--1005.


		\bibitem{FuLin} X. Fu and H. Lin,
		\emph{Nonamenable simple \(C^*\)-algebras with tracial approximation},
		Forum Math. Sigma \textbf{10} (2022),~e14, 1--50.
		
		
		\bibitem{GHV} E. Gardella, I. Hirshberg and A. Vaccaro,
		\emph{Strongly outer actions of amenable groups on \(\mathcal Z\)-stable
			nuclear \(C^*\)-algebras},
		J. Math. Pures Appl. (9) \textbf{162} (2022), 76--123.
		
		\bibitem{GL} E. Gardella and M. Lupini,
		\emph{Group amenability and actions on \(\mathcal Z\)-stable
			\(C^*\)-algebras},
		Adv. Math. \textbf{389} (2021),~107931.
		
		\bibitem{HaagerupHaagerupRamirez}
		S. Haagerup, U. Haagerup and M. Ramirez-Solano,
		\emph{A computational approach to the Thompson group~$F$},
		Internat. J. Algebra Comput. \textbf{25} (2015), no.~3, 381--432.
		
		\bibitem{Haagerup} U. Haagerup,
		\emph{Quasitraces on exact \(C^*\)-algebras are traces},
		C. R. Math. Rep. Acad. Sci. Canada \textbf{36} (2014), no.~2--3, 67--92.
		
		\bibitem{HaagerupOlesen} U. Haagerup and K. K. Olesen,
		\emph{Non-inner amenability of the Thompson groups $T$ and $V$},
		J. Funct. Anal. \textbf{272} (2017), no.~11, 4838--4852.
		
		\bibitem{HaagerupPicioroaga} U. Haagerup and G. Picioroaga,
		\emph{New presentations of Thompson's groups and applications},
		J. Operator Theory \textbf{66} (2011), no.~1, 217--232.
		
		\bibitem{HWZ} I. Hirshberg, W. Winter and J. Zacharias,
		\emph{Rokhlin dimension and \(C^*\)-dynamics},
		Comm. Math. Phys. \textbf{335} (2015), no.~2, 637--670.
		
		\bibitem{JiangSu} X. Jiang and H. Su,
		\emph{On a simple unital projectionless $C^*$-algebra},
		Amer. J. Math. \textbf{121} (1999), no.~2, 359--413.
		
		\bibitem{Jolissaint} P. Jolissaint,
		\emph{Central sequences in the factor associated with Thompson's group~\(F\)},
		Ann. Inst. Fourier \textbf{48} (1998), no.~4, 1093--1106.
		
		\bibitem{Jolissaint2005} P. Jolissaint,
		\emph{Operator algebras related to Thompson's group~$F$},
		J. Aust. Math. Soc. \textbf{79} (2005), no.~2, 231--241.
		
		\bibitem{KalantarScarparo} M. Kalantar and E. Scarparo,
		\emph{Boundary maps, germs and quasi-regular representations},
		Adv. Math. \textbf{394} (2022),~108130.
		
		\bibitem{Kesten} H. Kesten,
		\emph{Full Banach mean values on countable groups},
		Math. Scand. \textbf{7} (1959), 146--156.
		
		\bibitem{KimLodha}
		Sang-hyun Kim and Y. Lodha,
		\emph{A note on the proof of nonamenability of~F},
		arXiv:2610.18360v1 (2026),
		\url{https://arxiv.org/abs/2610.18360v1}.
		
		\bibitem{KirchbergRordamCharacters} E. Kirchberg and M. R{\o}rdam,
		\emph{When central sequence $C^*$-algebras have characters},
		Internat. J. Math. \textbf{26} (2015), no.~7,~1550049, 32~pp.
		
		\bibitem{LanceNuclear} C. Lance,
		\emph{On nuclear C*-algebras},
		J. Funct. Anal. \textbf{12} (1973), no.~2, 157--176.
		
		\bibitem{LLV26} C. Le Bars, E. Leibtag and I. Vigdorovich,
		\emph{Selflessness, MIF and opposition in groups acting on exotic buildings},
		arXiv:2609.17129v1 (2026), \url{https://arxiv.org/abs/2609.17129v1}.
		
		\bibitem{LBMB} A. Le Boudec and N. Matte Bon,
		\emph{Subgroup dynamics and \(C^*\)-simplicity of groups of homeomorphisms},
		Ann. Sci. \'Ec. Norm. Sup\'er. (4) \textbf{51} (2018), no.~3, 557--602.
		
		\bibitem{Ocneanu} A. Ocneanu,
		\emph{Actions of discrete amenable groups on von Neumann algebras},
		Lecture Notes in Math. \textbf{1138}, Springer-Verlag, Berlin, 1985.
		
		\bibitem{OpenAINonamenability}
		OpenAI,
		\emph{Thompson's group \(F\) is nonamenable},
		OpenAI Math Release preprint
		\href{https://github.com/openai/math/blob/1581d4cbdd69a0faf34b2089614fe2a814316a6a/preprints/Thompsons-group-F-is-nonamenable-September-23-2026/paper.pdf}{OAI:Thompsons-group-F-is-nonamenable-September-23-2026},
		2026.
		
		\bibitem{OpenAICharacter}
		OpenAI,
		\emph{The Kirchberg--R\o rdam character criterion},
		OpenAI Math Release preprint \href{https://github.com/openai/math/blob/1581d4cbdd69a0faf34b2089614fe2a814316a6a/preprints/The-Kirchberg-Rordam-character-criterion-September-25-2026/paper.pdf}{OAI:The-Kirchberg-Rordam-character-criterion-September-25-2026},
		2026.
		
		\bibitem{OsakaPhillips} H. Osaka and N. C. Phillips,
		\emph{Stable and real rank for crossed products by automorphisms
			with the tracial Rokhlin property},
		Ergodic Theory Dynam. Systems \textbf{26} (2006), no.~5, 1579--1621.
		
		\bibitem{Ozawa26} N. Ozawa,
		\emph{Proximality and selflessness for group $C^*$-algebras},
		arXiv:2508.07938 (2026), \url{https://arxiv.org/abs/2508.07938}.
		
		\bibitem{RobertSelfless} L. Robert,
		\emph{Selfless \(C^*\)-algebras},
		Adv. Math. \textbf{478} (2025),~110409.
		
		\bibitem{RobertComparison}
		L. Robert,
		\emph{Strict comparison and selflessness},
		arXiv:2609.15776v2, 23 September 2026.
		\url{https://arxiv.org/abs/2609.15776v2}.
		
		\bibitem{Rordam} M. R\o rdam,
		\emph{The stable and the real rank of \(\mathcal Z\)-absorbing
			\(C^*\)-algebras},
		Internat. J. Math. \textbf{15} (2004), no.~10, 1065--1084.
		
		\bibitem{Rybak26} E. Rybak,
		\emph{Boundary dynamics, triple transitivity, and mixed identities in
			weakly hyperbolic groups},
		arXiv:2605.14159v3 (2026), \url{https://arxiv.org/abs/2605.14159v3}.
		
		\bibitem{SatoAmenable} Y. Sato,
		\emph{Actions of amenable groups and crossed products of
			\(\mathcal Z\)-absorbing \(C^*\)-algebras},
		in \emph{Operator Algebras and Mathematical Physics},
		Adv. Stud. Pure Math. \textbf{80}, Math. Soc. Japan, Tokyo, 2019, 189--210.

		
		\bibitem{STW26} C. Schafhauser, A. Tikuisis and S. White,
		\emph{Nuclear $C^*$-algebras: 99 problems},
		M\"unster J. Math. (to appear), 2026,
		special volume dedicated to the memory of Eberhard Kirchberg.
		\url{https://arxiv.org/abs/2506.10902v2}.
		
		\bibitem{SzaboSSA} G. Szab\'o,
		\emph{Strongly self-absorbing \(C^*\)-dynamical systems},
		Trans. Amer. Math. Soc. \textbf{370} (2018), no.~1, 99--130.
		Corrigendum: Trans. Amer. Math. Soc. \textbf{373} (2020), no.~10,
		7527--7531.
		
		\bibitem{SzaboSI} G. Szab\'o,
		\emph{Equivariant property (SI) revisited},
		Anal. PDE \textbf{14} (2021), no.~4, 1199--1232.
		
		\bibitem{TomsWinter} A. S. Toms and W. Winter,
		\emph{Strongly self-absorbing \(C^*\)-algebras},
		Trans. Amer. Math. Soc. \textbf{359} (2007), no.~8, 3999--4029.
		
		\bibitem{Vig25} I. Vigdorovich,
		\emph{Structural properties of reduced $C^*$-algebras associated with higher-rank lattices},
		arXiv:2503.12737 (2025), \url{https://arxiv.org/abs/2503.12737}.
		
		\bibitem{Vig26} I. Vigdorovich,
		\emph{Selfless reduced $C^*$-algebras of linear groups},
		Proc. Lond. Math. Soc. (3) \textbf{133} (2026), no.~1,~e70180.


\end{thebibliography}
\end{document}